\documentclass[12pt]{amsart}

\usepackage{amsmath}
\usepackage[left=2.5cm,top=2.5cm,bottom=2.5cm,right=2.5cm,letterpaper]{geometry}
\usepackage{amsfonts}
\usepackage{amssymb}
\usepackage{amsthm}
\usepackage{graphicx}
\usepackage{gensymb}
\usepackage{enumerate}
\usepackage{mathrsfs}
\usepackage{rotating}

\usepackage{longtable}
\usepackage{tikz}
\usetikzlibrary{matrix}
\usepackage{arydshln}

\newcommand{\C}{\ensuremath{\mathbb{C}}}

\newcommand{\Z}{\ensuremath{\mathbb{Z}}}

\newcommand{\g}{\ensuremath{\mathfrak{g}}}		
\newcommand{\h}{\ensuremath{\mathfrak{h}}}	
\newcommand{\n}{\ensuremath{\mathfrak{n}}}

\newcommand{\pr}{\text{pr}}
\renewcommand{\epsilon}{\varepsilon}

\newcommand{\Lcal}{\mathcal{L}}

\usepackage{hyperref}

\newcommand{\z}{\mathfrak{z}}
\newcommand{\flUuv}{\mathfrak{l}_{V,u,v}}

\newcommand{\m}{\mathfrak{m}}
\newcommand{\fd}{\mathfrak{d}}
\newcommand{\CL}{\mathcal{CL}}
\newcommand{\ra}{\rangle}
\newcommand{\la}{\langle}
\newcommand{\fl}{\mathfrak{l}}

\renewcommand{\u}{\mathfrak{u}}
\newcounter{alistcounter}

\usepackage{blkarray}

\newcommand{\fr}{\mathfrak{r}}

\newtheorem{theorem}{Theorem}[section]

\newtheorem{prop}[theorem]{Proposition}
\newtheorem{lemma}[theorem]{Lemma}

\theoremstyle{definition}
\newtheorem{mydef}[theorem]{Definition}

\theoremstyle{remark}
\newtheorem{remark}[theorem]{Remark}

\newcommand{\fs}{\mathfrak{s}}

\usepackage{lscape}

\newcommand{\fc}{\mathfrak{c}}

\begin{document}

\title{Coisotropic Subalgebras of Complex Semisimple Lie Bialgebras}
\author{Nicole Kroeger}
\address{University of Notre Dame \\ 255 Hurley Hall \\ Notre Dame, IN 46556}
\curraddr{S.C. Governor's School for Science and Mathematics \\ 401 Railroad Ave.\\ Hartsville, SC 29550}
\email{nkroeger@alumni.nd.edu}

\begin{abstract}
	In his paper ``A Construction for Coisotropic Subalgebras of Lie Bialgebras'', Marco Zambon gave a way to use a long root of a complex semisimple Lie biaglebra $\g$ to construct a coisotropic subalgebra of $\g$ (\cite{Zambon}). In this paper, we generalize Zambon's construction. 
Our construction is based on the theory of Lagrangian subalgebras of the double $\g\oplus\g$ of $\g$, and our coisotropic subalgebras correspond to torus fixed points in the variety $\Lcal(\g\oplus\g)$ of Lagrangian subalgebras of $\g\oplus\g$. 
\end{abstract}
\maketitle

\section{INTRODUCTION}

In \cite{Zambon}, Marco Zambon gave a way of using a long root $\beta$ of a complex semisimple Lie biaglebra $\g$ to construct a coisotropic subalgebra of $\g$. In this paper, we generalize Zambon's construction. Let $w_0$ be the long element of the Weyl group, W, of $\g$. To each pair $u,v\in W$ such that $u\leq vw_0$ in the weak order, we construct a coisotropic subalgebra $\mathfrak{c}_{u,v}$.  This generalizes Zambon's construction as we show later.
Our construction is based on the theory of Lagrangian subalgebras of the double $\g\oplus\g$ of $\g$, and our coisotropic subalgebras $\fc_{u,v}$ correspond to torus fixed points in the variety $\Lcal(\g\oplus\g)$ of Lagrangian subalgebras of $\g\oplus\g$.  It would be interesting to understand further coisotropic subalgebras of $\g$ using the theory of Lagrangian subalgebras.

Let $\g$ be a complex semisimple Lie algebra. A Lie bialgebra structure on $\g$ is a map $\delta:\g\to \g\wedge\g$ such that $\delta$ is a 1-cocyle for $\g$ and the map dual to $\delta$ gives a Lie bracket on $\g^*$. If $\g$ is a Lie bialgebra, then $\g$ and $\g^*$ have compatible Lie brackets. 
If $\g$ is a Lie bialgebra, then a subalgebra $\m$ of $\g$ is called a \textit{coisotropic subalgebra} if the annihilator $\m^0$ of $\m$ in $\g^*$ is a Lie subalgebra of $\g$. For this paper, we will focus on the standard Lie bialgebra structure on $\g$, which corresponds to the so-called standard Poisson Lie group structure on $G$, the adjoint group of $\g$ (see \cite[Chapter 11]{PoissonBook}).

We will study coisotropic subalgebras via their connection to Lagrangian subalgebras.  Let $\mathfrak{d}$ be a $2n$-dimensional complex Lie algebra with a symmetric, non-degenerate, ad-invariant bilinear form $\langle \, , \,\rangle$.   A subalgebra $\fl$ of $\mathfrak{d}$ is called a  \textit{Lagrangian subalgebra} if $\dim \fl=n$ and $\fl$ is isotropic with respect to $\langle \, , \, \rangle$.  A cositropic subalgebra $\m$ of $\g$ gives rise to a Lagrangian subalgebra $\m\oplus\m^0$ of $\g\oplus\g^*\cong \g\oplus\g$.  The Lagrangian subalgebras of $\g\oplus\g$ have been classified by E. Karolinsky (see \cite{Karo}). Hence, one can study coisotropic subalegbras of $\g$ by studying certain Lagrangian subalgebras in $\g\oplus\g$.  For $\fl$ Lagrangian in $\g\oplus\g$, we say $\fl$ is coisotropic if $\fl$ can be written as $\m\oplus\m^0$ for $\m\subset \g$ coisotropic.

Let $\g=\n_-\oplus\h\oplus\n_-$ be the triangular decomposition.  Let $W$ be the Weyl group.  For a suspace $V$ of $\h$, let $V_\Delta=\{(x,x)|x\in V\}$ and $V_{-\Delta}=\{(x,-x)|x\in V\}$.  Define 
	\[
		\fl_{V,u,v}:=V_\Delta+(V^\perp)_{-\Delta}+(u\cdot\n, 0)+(0, v\cdot \n_-)
	\]
where $V\subseteq \h$, $u,v\in W$.  For all such triples $(V, u,v),$ $\fl_{V,u,v}$ is Lagrangian in $\g\oplus\g$.  In this paper, we determine when $\fl_{V,u,v}$ is coisotropic.  We prove the following two theorems which improve and clarify the results of Zambon. 

\begin{theorem}\label{intro theorem}
	For any $w\in W$, let $\Phi_w$ be the set of positive roots made negative by $w^{-1}.$ Then $\fl_{V,u,v}$ is coisotropic iff $\Phi_u\cap\Phi_v=\emptyset$.
\end{theorem}

For a root $\beta,$ let $s_\beta$ be the reflection through $\beta$ and $H_\beta\in \h$ be the unique element such that $\beta(H_\beta)=2$ and $H_\beta\in [\g_\beta, \g_{-\beta}]$. 

\begin{theorem}
	Zambon's coisotropic subalgebras are of the form
		$\fl_{\C H_\beta, s_\beta, e}$ or $\fl_{\C H_\beta, e, s_\beta}$
where $\beta\in\Phi^+$ is a long root.
	In particular, the coisotropic subalgebras described by Zambon are a specific example of the more general cosiotropic subalgebras described in Theorem \ref{intro theorem} with $u=e$ and $v=s_\beta$ or $u=s_\beta$ and $v=e$ respectively. 
\end{theorem}

In \cite{Zambon}, Zambon computes his coisotropic subalgebras for the classical Lie algebras. Our construction computes Zambon's coisotropic subalgebras for all complex semisimple Lie algebras using Chevalley bases.

Coisotropic subalgebras are interesting in part due to their relation to quantum homogeneous spaces. In particular, coisotropic subalgebras give rise to Poisson homogeneous spaces (see \cite{Drin}). In \cite{Ohayon}, Ohayon gives the quantization of the Poisson homogeneous spaces related to Zambon's coisotropic subalgebras in the classical cases. It would be interesting to quantize the coisotropic subalgebras described in Theorem \ref{intro theorem}.

 The results of this paper are from my PhD dissertation at the University of Notre Dame (see \cite{thesis}). I am especially thankful to Sam Evens for all his help and guidance with this work. Thank you also to Matthew Dyer for useful comments. The Arthur J. Schmitt Foundation partially supported this work.

\section{Preliminaries}

In this section, we recall the facts from the literature which will be necessary in later sections.

\subsection{Notation} For this paper, let $\g$ be a complex semisimple Lie algebra. Let $G$ be a connected Lie group with Lie algebra $\g$.  Fix a Cartan subalgebra $\h$ of $\g$ and a set $\Phi^+$ of positive roots in $\Phi$, the set of all roots of $(\g,\h)$.  Let $W$ denote the corresponding Weyl group. Write $\Phi^-$ for $-\Phi^+$, the negative roots of $\Phi$.  Let $\Delta$ be the set of simple roots in $\Phi^+$.  For $\alpha\in \Phi$, let $\g_\alpha$ be the root space corresponding to $\alpha$. Also,  let $H_\alpha\in\h$ be the unique element such that $\alpha(H_\alpha)=2$ and $H_\alpha\in [\g_\alpha, \g_{-\alpha}]$.  Let $\g=\n\oplus\h\oplus\n_-$ be the triangular decomposition of $\g$ where $\n$ and $\n_-$ are choosen so that the roots of $\n$ are positive roots.

For any subset $\u\subseteq \g$, we  write $\u_\Delta:=\{(x,x)|x\in\u\}$ in $\g\oplus\g$ and $\u_{-\Delta}:=\{(x,-x)|x\in\u\}$ in $\g\oplus\g$.  Let  $\{E_\alpha |\alpha\in\Phi\}\cup \{H_i|1\leq i\leq \text{rank}(\g)\}$ be a Chevalley basis for $\g$ where $E_\alpha\in \g_\alpha$ for all $\alpha\in \Phi$ (see \cite[\S VI.6]{Serre}). 

\subsection{Lie bialgebras}\label{Lie bialgebra section}
It will be necessary to study Lie algebras, $\fr$, such that $\fr^*$ has a Lie bracket compatible with the Lie bracket on $\fr^*$. For more details from this section, see for instance \cite{KoroSoibel}, \cite{PoissonBook}, or \cite{Yvette}.
\begin{mydef}
A \textit{Lie bialgebra} $(\fr,\phi)$ is a Lie algebra $\fr$ together with a map $\phi:\fr \to \fr\wedge \fr$ such that
	\begin{enumerate}
		\item The dual map $\phi^*:\fr^*\wedge \fr^* \to \fr^*$ is a Lie bracket on $\fr^*.$
		\item $\phi$ is a $1$-cocycle on $\fr$ where $\fr$ acts on $\fr\wedge \fr$ by the adjoint action.
		\end{enumerate}
\end{mydef}

Let $\fr$ be a Lie bialgebra.  If $\m\subset\fr$ is a Lie subalgebra of $\fr$ such that the annihilator $\m^0$ of $\m$ in $\fr^*$ is a Lie sublagebra of $\fr^*$, then $\m$ is called a \textit{coisotropic subalgebra} of $\fr$.

\subsection{Standard Lie bialgebra structure}
 Let $\fd$ be a finite dimensional Lie algebra with a nondegenerate invariant bilinear from. Let $\u_1,\u_2$ be Lie subalgebras of $\fd$.  The triple $(\fd,\u_1, \u_2)$ is called a \textit{Manin triple} if $\fd=\u_1\oplus\u_2$ as vector spaces and $\u_1, \u_2$ are maximal isotropic subspaces of $\fd$. 
	There is a one-to-one correspondence between Manin triples and Lie biaglebras (see \cite[Theorem 2.3.2]{KoroSoibel}).  In particular, if $(\fd, \u_1, \u_2)$ is a Manin triple, then $(\u_1, \phi)$ is a Lie bialgebra where $\phi$ is the dual map to the Lie bracket in $\u_2\cong \u_1^*$. 

We now provide an example of a Manin triple which will be useful throughout this paper.  Consider the Lie algebra $\g\oplus\g$ where $\g$ is our fixed semisimple Lie algebra.  
For $x_i, y_i\in\g$, define	
	\begin{equation}\label{bilinear form on g plus g} 
		\la (x_1,y_1), (x_2,y_2)\ra=\la\la x_1,x_2\ra\ra-\la\la y_1, y_2\ra\ra
	\end{equation}
 where $\langle\langle \, , \, \rangle\rangle$ is a fixed non-zero scalar multiple of the Killing form of $\g$. Then $\langle \, , \, \rangle$ is a symmetric, non-degenerate, and ad-invariant bilinear form on $\g\oplus\g$.				

\begin{lemma}\label{manin triple claim} \cite[Exercise 2.3.7]{KoroSoibel}
	Let $\g_\Delta=\{(x,x)|x\in\g\}$ and $\u_2=\h_{-\Delta} +(\n,0)+(0,\n_-)=\{(x+y, -x+z)|x\in \h, y\in \n, z\in \n_-\}$.  Then $(\g\oplus\g, \g_\Delta, \u_2)$ is a Manin triple where the bilinear form on $\g\oplus\g$ is given by equation (\ref{bilinear form on g plus g}).  
\end{lemma}

The Manin triple from Lemma \ref{manin triple claim} gives rise to a Lie bialgebra structure on $\g_\Delta\cong \g$.  We call this the \textit{standard Lie bialgebra structure} on $\g$.
Furthermore, since $(\g\oplus\g, \g_\Delta, \u_2)$ is a Manin triple, we have $ \u_2\cong (\g_\Delta)^*\cong \g^*$, i.e.,
	\[
		\g^*\cong \h_{-\Delta}+(\n,0)+(0,\n_-).
	\]
For the remainder of the paper, we will assume $\g$ has the standard Lie bialgebra structure. We identify $\g^*$ with $\h_{-\Delta}+(\n,0)+(0,\n_-)$ and $\g$ with $\g_\Delta$.

When $\g$ has the standard Lie bialgebra structure, the annihaltor of $\m$ in $\g^*$ is
	\[
		\m^\perp:=\{(x,y)\in \g^*| \langle(x,y), (z,z) \rangle=0 \text{ for all } z\in \m\}.
	\]
Therefore, $\m$ is a coisotropic subalgebra of $\g$ iff $\m^\perp$ is a Lie subalgebra of $\g^*$.

If $\m$ is a coisotropic subalgebra, then $\m_\Delta$ and $\m^\perp$ are subalgebras of $\g_\Delta$ and $\g^*$ respectively. One can show that $\m_\Delta\oplus\m^\perp$ is a subalgebra of $\g\oplus\g\cong \g_\Delta\oplus\g^*$.  Furthermore, $\dim(\m_\Delta\oplus\m^\perp)=\dim(\g)$ and it is easy to check that $\m_\Delta\oplus\m^\perp$ is isotropic in $\g\oplus\g$.  This gives the following lemma.

	\begin{lemma}\label{coiso is lagrangian first claim}
		If $\m\subset \g$ is a coisotropic subalgebra, then $\m_\Delta\oplus \m^\perp$ is a Lagrangian subalgebra of $\g\oplus \g.$
	\end{lemma}

	Let $\Lcal(\g\oplus\g)$ be the variety of Lagrangian subalgebras of $\g\oplus\g$. In \cite{Karo}, E. Karolinsky gave a classification of the elements in $\Lcal(\g\oplus\g)$. In order to understand coisotropic subalgebras of $\g$, we consider the Lagrangian subalgebras of $\g\oplus\g$ and decide when we can write them in the form $\m_\Delta\oplus\m^\perp$ where $\m$ is a coisotropic subalgebra. 
	
	$G\times G$ acts on $\Lcal(\g\oplus\g)$ by the adjoint action of $G\times G$ on $\g\oplus\g$. The following proposition will motivate our process of searching for coisotropic subalgebras. 
	
\begin{prop}\cite[Corollary 2.26]{ELII} \label{GB times Gb_ prop from ELII}
	For every Lagrangian subspace $U$ of $\h\oplus \h$, the $(G\times G)$-orbit through $U+(\n,0)+(0,\n_-)$ is Poisson submanifold isomorphic to $G/B\times G/B_-$ where the isomorphism is given by 	
		\begin{align*}
			 G/B\times G/B_-  &\to (G\times G)\cdot (U+(\n,0)+(0,\n_-)) \\
		   (g_1 B, g_2 B_-) &\mapsto(g_1,g_2)\cdot (U+(\n,0)+(0,\n_-)) .
		\end{align*}	  
	Furthermore, these are the only closed $(G\times G)$-orbits in $\Lcal(\g\oplus\g)$.
\end{prop}

Let $\fl(U):=U+(\n,0)+(0,\n_-)$.  There are many closed $G\times G$-orbits, $(G\times G)\cdot \fl(U)$, inside $\Lcal(\g\oplus\g)$.   In this paper, we focus on certain such orbits and determine when they give rise to coisotropic subalgebras of $\g$.

Define
	\[
		\mathcal{CL}(\g\oplus\g)=\{\fl\in \mathcal{L}(\g\oplus\g)|\fl=\m_\Delta\oplus \m^\perp \text{ for a coisotropic subalgebra } \m\subseteq \g\}.
	\]
If $\fl\in \CL(\g\oplus\g)$, we will write $\fc(\fl)$ to be the corresponding coisotropic subalgebra of $\g$ such that $\fl=\fc(\fl)_\Delta + \fc(\fl)^\perp$. 

The following Lemma gives us equivalent conditions for when $\fl\in\CL(\g\oplus\g)$ and will be useful for determining whether a Lagrangian subalgebra is coisotropic. The proof of the lemma is easy and is left to the reader.
	
\begin{lemma}\label{main corollary about dimension and coisotropic}
	Let $\fl$ be a Lagrangian subalgebra of $\g\oplus\g$. The following are equivalent:
		\begin{enumerate}
			\item $\fl\in\CL(\g\oplus\g)$ 
			\item $\dim\pr_{\g_\Delta}(\fl)+\dim\pr_{\g^*}(\fl)=\dim(\g)=\dim(\fl)$
			\item $\fl=(\fl\cap\g_\Delta)\oplus (\fl\cap \g^*).$
		\end{enumerate}
\end{lemma}

\begin{remark}\label{remark to add}
	Let $\fl\in \CL(\g\oplus\g)$ and  $\m\subset \g$ be such that $\m_\Delta=\fl\cap \g_\Delta$.  Then by the above Lemma, $\m$ is a coisotropic subalgebra of $\g$.
\end{remark}

\subsection{} For this paper, it will be necessary to understand certain subsets of positive roots. For $u\in W$, define 
	\[
		\Phi_u:=u(\Phi^-)\cap \Phi^+=\{\alpha\in\Phi^+| u^{-1}(\alpha)\in\Phi^-\}
	\]
This set consists of the positive roots made negative under $u^{-1}$.  Let $\Phi_u^c:=\Phi^+\setminus \Phi_u$. 
\begin{lemma}\label{Phi claim}
For all $u, v\in W$, the following are equivalent
	\begin{enumerate}
		\item $\Phi_u\cap\Phi_v=\emptyset$
		\item $\Phi_u\subset \Phi_v^c$
		\item $\Phi_v\subset \Phi_u^c$
		\item $\ell(u)+\ell(v)=\ell(u^{-1}v)$.
	\end{enumerate}
\end{lemma}
The first three equivalences of the lemma follow directly from the definition of $\Phi_u$.  The equivalence of (1) and (4) follows from Exercise 1.13 and Proposition 4.4.6 of \cite{Bjorner}.

Note that we can define a partial order on $W$.  In particular, for $u,v\in W$, say $u\leq v$ if $v=us_1\dots s_k$ for some simple reflections $s_i$ such that $\ell(us_1s_2\dots s_i)=\ell(u)+i$ for all $0\leq i \leq k$. This defines the \textbf{weak order} on $W$ (see \cite{Bjorner}). Let $w_0$ be the long element of the Weyl group. The following proposition gives an additional equivalent condition for $\Phi_u\cap\Phi_u=\emptyset$ using the weak order.
\begin{prop}\cite[Proposition 3.1.2]{Bjorner}\label{2.3}
	For $u,v\in W$, $u\leq vw_0$ in the weak order iff $\Phi_u\cap\Phi_v=\emptyset$.
\end{prop}

\section{THE SUBVARIETY $\CL(\g\oplus\g)$ OF $\mathcal{L}(\g\oplus\g)$}

  In this section, we study the subset $\CL(\g\oplus\g)$  of $\mathcal{L}(\g\oplus\g)$.  
We introduce the Poisson structure  $\Pi$ on $\Lcal(\g\oplus\g)$ and use it to give a necessary condition for $\fl$ to be in  $\CL(\g\oplus\g)$.  This condition will then allow us to explicitly describe certain coisotropic subalgebras in $\g$.

	\subsection{}
In \cite{ELI}, Evens and Lu show that if $\g\oplus\g=\fl_1\oplus\fl_2$ is a splitting of $\g\oplus\g$ with $\fl_1, \fl_2\in \Lcal(\g\oplus\g)$, then there exists a Poisson structure $\Pi_{\fl_1, \fl_2}$ on $\Lcal(\g\oplus\g)$.  We consider only the splitting $\g\oplus\g\cong \g_\Delta \oplus\g^*$ and let $\Pi$ be the corresponding Poisson structure.  The action of $G\times G$ on $\Lcal(\g\oplus\g)$ gives a Lie algebra anti-homomorphism $\kappa$ from $\g\oplus\g$ to the space of vector fields on $\Lcal(\g\oplus\g)$. 
 Let $\{e_1, \dots, e_n\}$ be a basis of $\g$ and $\{\eta_1, \dots \eta_n\}$ be a basis for $\g^*$ with $\langle e_i, \eta_j\rangle =\delta_{ij}$, and define
	\begin{equation}\label{R equation}
		R=\frac{1}{2} \sum_{i} \eta_i \wedge e_i.
	\end{equation}
Then $\Pi=\wedge^2\kappa (R)$.

 \subsection{}
  In this section, we prove that $\fl\in \CL(\g\oplus\g)$ implies $\Pi(\fl)=0$.  This will allow us to narrow down which Lagrangian subalgebras of $\g\oplus\g$ can possibly be in $\CL(\g\oplus\g)$.  In particular, if $\fl\in \Lcal(\g\oplus\g)$ and $\Pi(\fl)\neq 0$, then the results of this section show that $\fl\notin \CL(\g\oplus\g)$.
\begin{prop}	\label{If l coisotropic, then pi(l)=0}
	If $\fl\in\CL(\g\oplus\g)$, then $\Pi(\fl)=0.$
\end{prop}

We begin with some facts that will lead up to the proof of this proposition. For more details, see \S 2.2 of \cite{ELI}.  Let $\mathfrak{d}=\g\oplus\g^*\cong \g\oplus\g$ be the double of the Lie bialgebra $\g$ and $D$ be the adjoint group so that $\text{Lie}(D)=\g\oplus\g$. For any $\fl\in \Lcal(\g\oplus\g)$, let $\Pi(\fl)^\sharp$ be the linear map
		\[
			\Pi(\fl)^\sharp: T_\fl^* (D\cdot \fl) \to T_\fl(D\cdot \fl)
		\]
	given by $\Pi(\fl)^\sharp(\lambda)(\mu)=\Pi(\fl)(\lambda, \mu)$ for $\lambda, \mu\in T_\fl^*(D\cdot\fl)$. 

Let $\kappa:\mathfrak{d}\to T_\fl(D\cdot \fl)$ be the action map, which induces an isomorphism,  $\kappa_*: \mathfrak{d}/N_\mathfrak{d}(\fl)\to T_\fl(D\cdot \fl)$.  Furthermore, the transpose map $\kappa^*=(\kappa_*)^{tr}: T_\fl^*(D\cdot \fl)\to (\mathfrak{d}/N_{\mathfrak{d}}(\fl))^*\cong N_\mathfrak{d}(\fl)^\perp$ is an isomorphism 

		

For $R\in \bigwedge^2 \mathfrak{d}$ as in equation (\ref{R equation}), we have $R^\sharp:\mathfrak{d}^*\to \mathfrak{d}$ given by $R^\sharp(\lambda)(\mu)=R(\lambda, \mu)$ for $\lambda, \mu\in \mathfrak{d}^*$.  We can consider the restriction of this map $R^\sharp: N_\mathfrak{d}(\fl)^\perp\to \mathfrak{d}$.  Composition with the projection $\mathfrak{d}\to \mathfrak{d}/N_\mathfrak{d}(\fl)$, gives a map $N_\mathfrak{d}(\fl)^\perp\to \mathfrak{d}/N_{\mathfrak{d}}(\fl)$ which by abuse of notation we call $R^\sharp$.  
Putting these maps together gives the following commutative diagram.
\begin{center}
\begin{tikzpicture}
  \matrix (m) [matrix of math nodes,row sep=3em,column sep=4em,minimum width=2em]
  {
     \mathfrak{d}/N_{\mathfrak{d}}(\fl) & N_{\mathfrak{d}}(\fl)^\perp \\
     T_{\fl}(D\cdot \fl) & T_{\fl}^*(D\cdot \fl) \\};
  \path[-stealth]
    (m-1-1) edge node [left] {$\kappa_*$} (m-2-1)
     (m-1-2)       edge  node [above] {$R^\sharp$} (m-1-1)
    (m-2-2) edge node [below] {$\Pi(\fl)^\sharp$}
            node [above] {} (m-2-1)
    (m-2-2) edge node [right] {$\kappa^*=\kappa_*^{tr}$} (m-1-2);
\end{tikzpicture}
\end{center}


\begin{lemma}\label{claim Pi for next Lemma}
	Let $x+\xi\in N_{\mathfrak{d}}(\fl)^\perp$ with $x\in \g$ and $\xi\in\g^*$. Then 
	\[
		\Pi(\fl)^\sharp((\kappa^*)^{-1}(x+\xi))=\kappa_*(x+N_{\mathfrak{d}}(\fl)).
	\]
\end{lemma}

\begin{proof}
	
	Since $\mathfrak{d}^*\cong \g^*\oplus\g$,
		\begin{equation*}
			R^\sharp (\xi+x)= \frac{1}{2}[\sum_j\langle x, \eta_j\rangle e_j - \sum_j \langle e_j, \xi\rangle \eta_j]\\
						= \frac{1}{2}(x-\xi)+N_\mathfrak{d}(\fl).
		\end{equation*}
	Now, $\fl$ is a Lagrangian subalgebra of $\g\oplus\g$, so $N_{\mathfrak{d}}(\fl)^\perp \subseteq \fl^\perp=\fl\subseteq N_{\mathfrak{d}}(\fl)$.  In particular, $x+\xi\in N_{\mathfrak{d}}(\fl)^\perp\subset N_{\mathfrak{d}}(\fl)$. Therefore,
		\begin{align*}
			R^\sharp(x+\xi) & =\frac{1}{2}(x-\xi)+\frac{1}{2}(x+\xi)+N_{\mathfrak{d}}(\fl) \\
							& = x+N_{\mathfrak{d}}(\fl).
		\end{align*}
		Finally, by the above commutative diagram, we have $\Pi(\fl)^\sharp=\kappa_*R^\sharp\kappa^*$. This completes the proof.
	
\end{proof}

Using the above notation, we now prove Proposition \ref{If l coisotropic, then pi(l)=0}.
\begin{proof}[Proof of Proposition \ref{If l coisotropic, then pi(l)=0}]
	Let $\fl\in \CL(\g\oplus\g)$ and let $x+\xi\in N_\mathfrak{d}(\fl)^\perp$ with $x\in\g$ and $\xi\in\g^*$. Then by Lemma \ref{claim Pi for next Lemma},  $\Pi(\fl)^\sharp ((\kappa^*)^{-1}(x+\xi))=\kappa_*(x+N_\mathfrak{d}(\fl)).$   
	
	$\fl$ is Lagrangian, so $N_\mathfrak{d}(\fl)^\perp\subseteq \fl$, and thus $x+\xi\in\fl$. Since $\fl\in\CL(\g\oplus\g),$   Lemma \ref{main corollary about dimension and coisotropic} gives $\fl=(\fl\cap \g)\oplus(\fl\cap\g^*)$, so $x,\xi\in\fl\subset N_\mathfrak{d}(\fl)$.
Therefore, 
		\[
			\Pi(\fl)^\sharp((\kappa^*)^{-1}(x+\xi))=\kappa_*(x+N_\mathfrak{d}(\fl))=\kappa_*(0)=0.
		\]
	  It follows that $\Pi(\fl)=0$.
\end{proof}

\subsection{Rank of $\Pi(\fl)$}
The following lemma will enable us to identify some subalgebras $\fl$ such that $\Pi(\fl)=0$.

From Proposition \ref{GB times Gb_ prop from ELII}, we can identify the $G\times G$-orbit through $U+(\n,0)+(0,\n_-)$ where $U\in \Lcal_{\text{\space}}(\h\oplus\h)$ with $G/B\times G/B_-$.  Therefore, we view $\Pi$ as a Poisson structure on $G/B\times G/B_-$.  
\begin{lemma}\cite[Example 4.9]{ELII}\label{General Formula for Pi Lemma}
	Let $w,w_1,w_2\in W$. If $\fl\in G_\Delta \cdot (eB, wB_-)\cap G^*\cdot(w_1B, w_2B_-),$ then
		\[
			\text{rk}(\Pi(\fl))=\ell(w_1)+\ell(w_2)-\ell(w)-\dim(\h^{-ww_2^{-1}w_1}).
		\]
\end{lemma}

The following proposition will be useful in the subsequent sections to determine some specific coisotropic subalgebras of $\g$.

\begin{prop}\label{Zuv Pi Lemma}
For $u,v\in W$, let $\fl\in G_\Delta\cdot(uB, vB_-)\cap G^*\cdot (uB, vB_-)$.  We have the following.
\begin{enumerate}
	\item[(1)] $\text{rk}(\Pi(\fl))=\ell(u)+\ell(v)-\ell(u^{-1}v)$
	\item[(2)] $\text{rk}(\Pi(\fl))=0$ iff $\Phi_u\cap \Phi_v=\emptyset$
	\item[(3)] If $\fl\in\CL(\g\oplus\g)$, then $\Phi_u\cap\Phi_v=\emptyset$.
\end{enumerate}
			
\end{prop}

\begin{proof}
 For (1), 	by assumption $\fl\in G_\Delta \cdot (uB, vB_-)=G_\Delta \cdot (eB, u^{-1}vB_-)$ and $\fl\in G^*\cdot (uB, vB_-)$.  Therefore, following the notation of Lemma \ref{General Formula for Pi Lemma}, we have $w=u^{-1}v$, $w_1=u$, and $w_2=v$.  Applying Lemma \ref{General Formula for Pi Lemma} gives $\text{rk}(\Pi(\fl))=\ell(u)+\ell(v)-\ell(u^{-1}v)-\dim(\h^{-u^{-1}vv^{-1}u})$.  However, $\h^{-u^{-1}vv^{-1}u}=\h^{-\text{id}}=0$, and the result follows. 
 
(2) follows from Lemma \ref{Phi claim}.
 Finally, (3) follows from (2) and Proposition \ref{If l coisotropic, then pi(l)=0}.

\end{proof}

\begin{remark}
	The converse of Proposition \ref{If l coisotropic, then pi(l)=0} is false.  For example, let $\z_{u,v}:=(u,v)\cdot[\h_\Delta +(\n,0)+(0,\n_-)]$.  One can show that $\z_{u,v}\in \CL(\g\oplus\g)$ iff $\Phi_u\cap\Phi_v=\emptyset$ and $(v^{-1}u)^2=e$. In particular, if $u\in W$ such that $u^2\neq e$, then $\z_{u,e}\notin \CL(\g\oplus\g)$ by the above remarks, but by Proposition \ref{Zuv Pi Lemma}(2), $\Pi(\z_{u,e})=0$. See \cite{thesis} for more details.
\end{remark}


\subsection{A Family of Coisotropic Subalgebras}\label{section lUvu}

We use Proposition \ref{If l coisotropic, then pi(l)=0} and further structure theory to determine when a class of Lagrangian subalgebras are coisotropic.

\begin{theorem}\label{Main Lemma about V+(un,0)+(0,vn-) being coisotropic}
	Let $V$ be a subspace of $\h$ and $u,v\in W$, and define 
		\[
			\flUuv:=V_\Delta \oplus (V^\perp)_{-\Delta} +(u\cdot \n, 0) + (0, v\cdot \n_-).
		\]  
	Then $\fl_{V,u,v}\in\CL(\g\oplus\g)$ iff $\Phi_u\cap\Phi_v=\emptyset$. 
	Furthermore, if $\flUuv\in\CL(\g\oplus\g)$, 
		\[
			\fc(\fl_{V,u,v})=V+\bigoplus\limits_{\alpha\in\Phi_u} \g_{-\alpha}  + \bigoplus\limits_{\alpha\in\Phi_v} \g_\alpha
		\]
	is the corresponding coisotropic subalgebra of $\g$. 
\end{theorem}

\begin{proof}
Note that $\fl_{V,u,v}=(u,v)\cdot [(u^{-1}, v^{-1})U+(\n,0)+(0,\n_-)]$ where $U=V_\Delta+(V^\perp)_{-\Delta}.$  
In the identification $(G\times G)\cdot \fl_{V,u,v}\cong G/B\times G/B_-$, $\fl_{V,u,v}\in\Lcal(\g\oplus\g)$ corresponds to $(uB, vB_-)$ in $G/B\times G/B_-$.  
Therefore,  $\flUuv\in G_\Delta \cdot (uB, vB_-) \cap G^*\cdot (uB, vB_-)$. By Proposition \ref{Zuv Pi Lemma}, if $\flUuv\in\CL(\g\oplus\g),$ then $\Phi_u\cap \Phi_v=\emptyset.$  

Now, assume $\Phi_u\cap \Phi_v=\emptyset$. We use Lemma \ref{main corollary about dimension and coisotropic} to show $\flUuv\in\CL(\g\oplus\g)$. 	
	For any $u\in W$, $u\Phi^+=\Phi_u^c \coprod -\Phi_u$. Hence,
			\[(u\cdot \n, 0)=\bigoplus\limits_{\alpha\in \Phi_u^c} (\g_\alpha,0) + \bigoplus\limits_{\alpha\in \Phi_u} (\g_{-\alpha},0)=\bigoplus\limits_{\alpha\in \Phi_u^c} (\g_\alpha,0) + \bigoplus\limits_{\alpha\in \Phi_u} \C\cdot[(E_{-\alpha},E_{-\alpha})-(0,E_{-\alpha})]
			\]
		and
		\[
		(0,v\cdot\n_-)=\bigoplus\limits_{\alpha\in \Phi_v} (0,\g_\alpha) + \bigoplus\limits_{\alpha\in \Phi_v^c} (0,\g_{-\alpha})=\bigoplus\limits_{\alpha\in \Phi_v} \C\cdot[(E_{\alpha},E_{\alpha})-(E_{\alpha},0)] + \bigoplus\limits_{\alpha\in \Phi_v^c} (0,\g_{-\alpha}).
		\]
		Therefore, 
			\[
				\pr_{\g_\Delta}(\flUuv)=V_\Delta+ \bigoplus\limits_{\alpha\in \Phi_u} (\g_{-\alpha})_\Delta + \bigoplus\limits_{\alpha\in \Phi_v} (\g_\alpha)_\Delta
			\]
		and 
			\[
				\pr_{\g^*}(\flUuv)=(V^\perp)_{-\Delta} + \bigoplus\limits_{\alpha\in \Phi_u^c}(\g_\alpha, 0)+\bigoplus\limits_{\alpha\in \Phi_v}(\g_\alpha, 0) + \bigoplus\limits_{\alpha\in\Phi_u}(0,\g_{-\alpha}) + \bigoplus\limits_{\alpha\in\Phi_v^c}(0,\g_{-\alpha}).
			\]
			
		Since $\Phi_u\cap \Phi_v=\emptyset$, Lemma \ref{Phi claim} implies that $\Phi_u\subset \Phi_v^c$ and $\Phi_v\subset \Phi_u^c$. Therefore, 
			\[
				\pr_{\g^*}(\flUuv)=(V^\perp)_{-\Delta} + \bigoplus\limits_{\alpha\in \Phi_u^c}(\g_\alpha, 0) +  \bigoplus\limits_{\alpha\in\Phi_v^c}(0,\g_{-\alpha}).
			\]
		Since $V\subset \h$, $\dim(V_\Delta)+\dim(V^\perp)_{-\Delta}=\dim(V)+\dim(V^\perp)=\dim(\h)$.  Thus 
			\begin{align*}
				\dim(\pr_{\g_\Delta}(\flUuv))+\dim(\pr_{\g^*}(\flUuv)) & =\dim(\h)+|\Phi_u|+|\Phi_v|+|\Phi_u^c|+|\Phi_v^c|\\
					& =\dim(\h)+2|\Phi^+|=\dim(\g).
			\end{align*}
		  Therefore, by Lemma \ref{main corollary about dimension and coisotropic}, $\flUuv\in\CL(\g\oplus\g)$ which completes the first part of the Theorem.
		
		Now, assume $\fl_{V,u,v}\in \CL(\g\oplus\g)$, so $\Phi_u\cap\Phi_v=\emptyset$ by Proposition \ref{Zuv Pi Lemma}.  Let $\m:=V+\oplus_{\alpha\in\Phi_u}\g_{-\alpha} + \oplus_{\alpha\in\Phi_v}\g_\alpha$. It is easy to check that 
			\[
				\m_\Delta=V_\Delta+ \bigoplus\limits_{\alpha\in \Phi_u} (\g_{-\alpha})_\Delta + \bigoplus\limits_{\alpha\in \Phi_v} (\g_\alpha)_\Delta
			\]
		and $\m^\perp$ in $\g^*$ is given by 
			\[
				\m^\perp=(V^\perp)_{-\Delta} + \bigoplus\limits_{\alpha\in \Phi_u^c}(\g_\alpha, 0) +  \bigoplus\limits_{\alpha\in\Phi_v^c}(0,\g_{-\alpha}).
			\]
		Therefore, $\m_\Delta+\m^\perp=\fl_{V,u,v}\in\CL(\g\oplus\g)$ and indeed $\m=\fc(\fl_{V,u,v})$, the coisotropic subalgebra of $\g$ corresponding to $\g$.
		
\end{proof}

\begin{remark}\label{Remark about how lUuv is split up!}
	By the proof of Theorem \ref{Main Lemma about V+(un,0)+(0,vn-) being coisotropic}, if $\fl_{V,u,v}\in\CL(\g\oplus\g)$, then  $\fl_{V,u,v}=\m_\Delta \oplus \m^\perp$ where
		\begin{equation}
			\m_\Delta=V_\Delta +\bigoplus_{\alpha\in\Phi_u}(\g_{-\alpha})_\Delta +\bigoplus_{\alpha\in\Phi_v} (\g_{\alpha})_\Delta
		\end{equation}
	and
		\begin{equation}
			\m^\perp= (V^\perp)_{-\Delta} + \bigoplus_{\alpha\in\Phi_u^c}(\g_\alpha, 0) + \bigoplus_{\alpha\in\Phi_v^c}(0,\g_{-\alpha}).
		\end{equation}
		
\end{remark}

Let $H_\Delta=\{(x,x)|x\in H\}$.  Note that $(G/B\times G/B_-)^{H_\Delta}=\{(uB, vB_-)|u,v\in W\}$ and $(uB, vB_-)$ corresponds to $(u,v)\cdot (U+(\n,0)+(0,\n_-))\in \Lcal(\g\oplus\g)$ for any Lagrangian sub\textit{space} $U$ of $\h\oplus\h$.  Therefore, whether $(uB, vB_-)\in \CL(\g\oplus\g)$ depends not only on $u,v$ but on $U$ also.

For $V, u,v$ as in Theorem \ref{Main Lemma about V+(un,0)+(0,vn-) being coisotropic}, let
	\[
		\mathfrak{s}_{V,u,v}:=V_\Delta + (V^\perp)_{-\Delta} +(u\cdot \n, 0)+(0, v\cdot \n).
	\]
Note that $\fs_{V,u,v}=\fl_{V,u,vw_0}$ where $w_0$ is the long element of the Weyl group.  Combining Proposition \ref{2.3} and Theorem \ref{Main Lemma about V+(un,0)+(0,vn-) being coisotropic}, we have $\fs_{V,u,v}\in \CL(\g\oplus\g)$ iff $u\leq v$ in the weak Bruhat order. This makes determining all pairs $u,v$ such that $\fs_{V,u,v}\in \CL(\g\oplus\g)$ into a tabulation for each $v$ of the $u$ such that $u\leq v$. This is standard to do.

\section{THE COISOTROPIC SUBALGEBRAS OF ZAMBON}\label{Zambon section}
In this section, we show that Zambon's coisotropic subalgebras are a special case of the coisotropic subalgebras described in Theorem \ref{Main Lemma about V+(un,0)+(0,vn-) being coisotropic} and compute his coisotropic subalgebras in more generality.  We begin by restating the main result of \cite{Zambon}. 
\begin{theorem}\cite[Proposition 4.3]{Zambon} \label{Zambon's Theorem}
	Let $\g$ be a complex semisimple Lie algebra with Killing form $K(\, ,\,)$.  Let $\beta \in\Phi^+$ be such that for all $\alpha\in\Phi$, $(\alpha+\Z\beta)\cap \Phi$ does not contain a string of three consecutive elements.  Then
		\begin{center}
			$\u_\beta:=[E_\beta, \pi]^\sharp \g^*$ and $\u_{-\beta}:=[E_{-\beta}, \pi]^\sharp \g^*$
		\end{center}
	are coisotropic subalgebras of $\g$ where 
		\[
			\pi=\sum_{\alpha\in\Phi^+} \lambda_\alpha 
		E_\alpha\wedge E_{-\alpha}
		\]
	with $\lambda_\alpha=\frac{1}{K(E_\alpha, E_{-\alpha})}$.  We refer to such $\u_\beta$ and $\u_{-\beta}$ as \textbf{Zambon coisotropic subalgebras}.
\end{theorem}

In \cite{Zambon}, the Zambon coisotropic subalgebras are determined for the classical Lie algebras.  As a consequence of Theorem \ref{theorem with calculation for ubeta}, we are able to determine the Zambon coisotropic subalgebrs for every simple Lie algebra.

\subsection{Long Roots} We need to understand which roots $\beta\in\Phi$ satisfy the condition that $(\alpha\cap \Z\beta)\cap \Phi$ does not contain 3 consecutive elements for all $\alpha\in\Phi$.

\begin{mydef}
	A root $\beta\in \Phi$ is called a \textbf{long root} if there is no $\alpha\in\Phi$ in the same irreducible component of $\Phi$ as $\beta$ such that $(\beta, \beta) < (\alpha, \alpha)$. Otherwise, $\beta$ is called a \textbf{short root}.
\end{mydef}

We show that the long roots are exactly those that satisfy
 the above condition.

\begin{lemma}\label{irreducible root sys}
	If $\Phi$ is an irreducible root system  in $E$ and $\beta$ is a short root, then there is some long root $\alpha\in\Phi$ such that $(\alpha, \beta)\neq 0$.
\end{lemma}

\begin{proof}
	Let $\gamma$ be a long root. Then $\sigma(\gamma)$ is a long root for all $\sigma\in W$ since the Weyl group preserves the length of roots.
	Furthermore,  $\{\sigma(\gamma)|\sigma\in W\}$ spans $E$ (See \cite[\S 10.4 Lemma B]{Humphreys}).
	Therefore, there is some $\sigma\in W$ such that $\alpha:=\sigma(\gamma)$ is not orthogonal to $\beta$.    
\end{proof}

%
%

\begin{prop}\label{long roots satisfy zambon's hypothesis}
	Let $\beta\in\Phi$. For all $\alpha\in\Phi$,  $(\alpha+\Z \beta)\cap \Phi$ does not contain a string of three consecutive elements iff $\beta$ is a long root.
\end{prop}

\begin{proof}
	$(\Leftarrow)$ If $\beta$ is a long root, then $\frac{2(\beta, \alpha)}{(\beta, \beta)}=0, \pm 1$ for all $\alpha\in \Phi$.  Therefore, the length of the $\alpha$-root string through $\beta$ is at most 2. 
	
	$(\Rightarrow)$ Let $\beta\in\Phi$ be such that $(\alpha+\Z \beta)\cap \Phi$ does not contain a string of three consecutive elements for all $\alpha$ and suppose for a contradiction that $\beta$ is not a long root. Hence, there exists $\alpha$ in the same irreducible component as $\beta$ such that 
	\begin{equation}\label{equation for once}
		(\beta, \beta) <(\alpha, \alpha).
	\end{equation}
	By Lemma \ref{irreducible root sys}, we can assume $\alpha$ and $\beta$ are not orthogonal.  Therefore, replacing $\alpha$ with $-\alpha$ if necessary, $s_\beta(\alpha)=\alpha+k\beta$ where $k>0$.  But equation (\ref{equation for once}) implies that $k\geq 2$.  Therefore, $\alpha+2\beta$ or $\alpha+3\beta$ is a root and $\{\alpha, \alpha+\beta, \alpha+2\beta\}$ is a root string of three elements, a contradiction.
	
\end{proof}
 Zambon's coisotropic subalgebras are all of the form $[E_\beta, \pi]^\sharp\g^*$ where $\beta$ is a long root. In the next several sections, we explicitly compute $[E_\beta, \pi]^\sharp\g^*$.

\subsection{Computing $[E_\beta, \pi]$}
In order to understand Theorem \ref{Zambon's Theorem}, we now compute $[E_\beta, \pi]$ where $\beta$ is a long root.  We begin with several necessary facts.

\begin{lemma}\label{lambda claim}
	Let $\lambda_\alpha:=\frac{1}{K(E_\alpha, E_{-\alpha})}.$
	If $\beta\in\Phi$ is a long root and $\alpha\in\Phi^+$, then $\lambda_\alpha=\lambda_{s_\beta(\alpha)}$. In particular,
	\begin{enumerate}
		\item If $\alpha+\beta$ is a root, then $\lambda_\alpha=\lambda_{\alpha+\beta}$.
		\item If $\alpha-\beta$ is a root, then $\lambda_\alpha=\lambda_{\alpha-\beta}$.
\end{enumerate}	 
\end{lemma}

Since the Killing form is a symmetric bilinear form, we have $\lambda_\gamma=\lambda_{-\gamma}$ for all $\gamma\in \Phi$.

\begin{proof} We prove (1) and (2) follows in the same fashion. Since $\{E_\alpha|\alpha\in\Phi\}\cup \{H_i\}$ is a Chevalley basis for $\g$,  $[E_\alpha, E_{-\alpha}]=H_\alpha$ for all $\alpha\in \Phi^+$.
	Since $\beta$ is a long root and $\alpha+\beta$ is a root, $s_\beta(\alpha)=\alpha+\beta$. Therefore, $(\alpha, \alpha)=(s_\beta(\alpha), s_\beta(\alpha))=(\alpha+\beta, \alpha+\beta)$. 
	
	For $\gamma\in\Phi$, let $t_\gamma\in \h$ be the unique element satisfying $\gamma(h)=K(t_\gamma, h)$ for all $h\in\h$. We have $t_\alpha=\frac{H_\alpha}{K(E_\alpha, E_{-\alpha})}$ and  $t_\alpha=\frac{(\alpha, \alpha)}{2} H_\alpha$ (See \cite[Proposition 8.3]{Humphreys}).  Therefore, $\frac{(\alpha, \alpha)}{2}=\frac{1}{K(E_{\alpha}, E_{-\alpha})}=\lambda_\alpha$.  Hence
		\[
			\lambda_\alpha=\frac{(\alpha, \alpha)}{2}=\frac{(\alpha+\beta, \alpha+\beta)}{2}=\frac{1}{K(E_{\alpha+\beta}, E_{-\alpha-\beta})}=\lambda_{\alpha+\beta}.
		\]
\end{proof}

\begin{lemma}\label{c claim}
	Let $\beta$ be a long root and $\alpha\in \Phi$ such that $\alpha+\beta$ is a root. If $c_{\alpha, \beta}\in \Z$ is such that $[E_\alpha, E_\beta]=c_{\alpha, \beta} E_{\alpha+\beta}$, then
		\[
			c_{\beta,\alpha}=-c_{\beta, -\alpha-\beta}.
		\]
 Furthermore, if $\alpha-\beta$ is a root, then 
		\[
			c_{\beta, -\alpha}= -c_{\beta, \alpha-\beta}.
		\]
	
\end{lemma}

\begin{proof}
	Assume $\beta$ is a long root and $\alpha+\beta$ is a root, then the $\beta$-string through $\alpha$ is $\alpha$, $\alpha+\beta$.  Furthermore, $s_\beta(\alpha)=\alpha+\beta$. 
	By Lemma 25.2 of \cite{Humphreys},
		\[
			[E_{-\beta}, [E_\beta, E_\alpha]]=E_\beta.
		\]
	On the other hand, by definition of $c_{\beta,\alpha}$, we have $[E_{-\beta}, [E_\beta, E_\alpha]]=[E_{-\beta}, c_{\beta, \alpha}E_{\alpha+\beta}]=c_{\beta, \alpha} c_{-\beta, \alpha+\beta} E_{\beta}$.  Furthermore, by definition of a Chevalley basis, $c_{-\beta, \alpha+\beta}=-c_{\beta, -\alpha-\beta}.$
	Therefore,
		\begin{equation}\label{product of c's}
			1=-c_{\beta, \alpha} c_{\beta, -\alpha-\beta}.
		\end{equation}
Since the $c_{\beta,\alpha}$ are integers, we have $c_{\beta,\alpha}=-c_{\beta, -\alpha-\beta}$. 	The second statement follows in a similar fashion.
\end{proof}

For the remainder of this section, let $\beta\in\Phi^+$ be a fixed long root. 
We are now in a position to compute $[E_\beta, \pi]$.  First, let 
	\[
		Q_+:=\{\alpha\in\Phi^+|s_\beta(\alpha)\in\Phi^+; s_\beta(\alpha) \neq \alpha\},
	\] 
	\[
		Q_0:=\{\alpha\in\Phi^+| s_\beta(\alpha)=\alpha\},
	\] and 
	\[
		Q_-:=\{\alpha\in\Phi^+|s_\beta(\alpha)\in\Phi^-\}.
	\] 
 Then  $\Phi^+= Q_+\cup Q_0\cup Q_-$ is a disjoint union. Furthermore, $Q_-=\Phi_{s_\beta}$. Let $\pi_+:=\sum_{\alpha\in Q_+} \lambda_\alpha E_\alpha \wedge E_{-\alpha}$ and define $\pi_0$ and $\pi_-$ similarly.  Then, 
	\begin{equation}
		\pi  = \sum\limits_{\alpha\in\Phi^+} \lambda_\alpha E_\alpha\wedge E_{-\alpha}=\pi_+ + \pi_0 + \pi_-.
	\end{equation}
		Therefore,
			\begin{equation}\label{general Ebeta pi equation}
				[E_\beta, \pi]=[E_\beta, \pi_+]+[E_\beta, \pi_0]+[E_\beta, \pi_-].
			\end{equation}
	If $\alpha\in Q_0$, then $s_\beta(\alpha)=\alpha$ and $\alpha\pm \beta\notin\Phi$. It follows easily that $[E_\beta, \pi_0]=0$. It remains to compute $[E_\beta, \pi_+]$ and $[E_\beta, \pi_-]$.
	
	\begin{lemma}\label{pi+ part is 0}
		$[E_\beta, \pi_+]=0.$
	\end{lemma}
	
	\begin{proof}
Note that $\alpha\in Q_+$ implies $s_\beta(\alpha)\in Q_+$.  In particular, $s_\beta$ acts on the elements of $Q_+$ and the orbits are of the form $\{\alpha, s_\beta(\alpha)\}$.  

Let $\alpha\in Q_+.$ Without lose of generality, we can assume $s_\beta(\alpha)=\alpha+\beta$.  Indeed, if $s_\beta(\alpha)\neq \alpha+\beta$, then $s_\beta(\alpha)=\alpha-\beta$.  In this case, replace $\alpha$ with $\alpha-\beta$.

By Lemmas \ref{lambda claim} and \ref{c claim}, 
			\begin{align*}
			[E_\beta, \lambda_\alpha E_\alpha \wedge E_{-\alpha} + \lambda_{s_\beta(\alpha)} E_{s_\beta(\alpha)}  \wedge E_{-s_\beta(\alpha)} ] & = \lambda_\alpha [E_\beta, E_\alpha \wedge E_{-\alpha} + E_{\alpha+ \beta} \wedge E_{-\alpha-\beta}] \\
			& = \lambda_\alpha (c_{\beta, \alpha}+c_{\beta, -\beta-\alpha}) E_{\alpha+\beta} \wedge E_{-\alpha} \\
			& = 0.
			\end{align*}
		Hence, $[E_\beta, \pi_+]=0$.	
		
	\end{proof}
	
Combining equation (\ref{general Ebeta pi equation}) with Lemma \ref{pi+ part is 0} gives
	\[
		[E_\beta, \pi]=[E_\beta, \pi_-].
	\]
Thus it only remains to compute $[E_\beta, \pi_-].$

Recall, $Q_-=\Phi_{s_\beta}$, so $\pi_- = \sum_{\alpha\in\Phi_{s_\beta}} \lambda_\alpha E_\alpha \wedge E_{-\alpha}$.  Note  $\beta\in\Phi_{s_\beta}$ and 
	\begin{equation} \label{Ebeta Hbeta equation}
		[E_\beta, \lambda_\beta E_\beta \wedge E_{-\beta}] = \lambda_\beta E_\beta \wedge H_\beta. 
	\end{equation}
	For $\alpha\in \Phi_{s_\beta}\setminus \{\beta\}$, $s_\beta(\alpha)=\alpha-\beta<0$ and  $\alpha+\beta$ is not a root.
Note that $-s_\beta$ acts on $\Phi_{s_\beta}\setminus \{\beta\}$.  The orbits of $-s_\beta$ on 	$\Phi_{s_\beta}\setminus \{\beta\}$ are of the form $\{\alpha, \beta-\alpha\}$ where $\alpha\in\Phi_{s_\beta}\setminus \{\beta\}$.  In particular, the elements of $\Phi_{s_\beta}\setminus \{\beta\}$ come in pairs. 
	
	  Therefore, if $\alpha\in \Phi_{s_\beta}\setminus\{\beta\}$, 
	\begin{align}\label{First Ebeta pi- equation}
	\begin{split}	 
		 [E_\beta, \lambda_\alpha E_\alpha \wedge E_{-\alpha} + \lambda_{\beta-\alpha} E_{\beta-\alpha}\wedge E_{\alpha-\beta}] & = \lambda_\alpha c_{\beta, -\alpha} E_\alpha \wedge E_{\beta-\alpha} + \lambda_{\beta-\alpha} c_{\beta, \alpha-\beta} E_{\beta-\alpha} \wedge E_\alpha \\
		 & = (\lambda_\alpha c_{\beta, -\alpha} -\lambda_{\beta-\alpha} c_{\beta, \alpha-\beta})E_\alpha \wedge E_{\beta-\alpha}
		\end{split}
	\end{align}
 where $\beta-\alpha>0$.

Since the Killing form is a symmetric bilinear form, Lemma \ref{lambda claim} gives $\lambda_{\alpha}=\lambda_{\beta-\alpha}$.  Furthermore, by Lemma \ref{c claim}, we have $c_{\beta, -\alpha} =-c_{\beta, \alpha-\beta}$.  In particular, $\lambda_\alpha c_{\beta, -\alpha} -\lambda_{\beta-\alpha} c_{\beta, \alpha-\beta}=2\lambda_\alpha c_{\beta, -\alpha}$.

Combining equations (\ref{Ebeta Hbeta equation}) and (\ref{First Ebeta pi- equation}), gives 
		\[
			[E_\beta, \pi_-]=E_\beta \wedge H_\beta + \sum\limits_{\substack{\alpha\in\Phi_{s_\beta}\\ \alpha\neq \beta}}\lambda_\alpha c_{\beta, -\alpha} E_\alpha \wedge E_{\beta-\alpha}.
		\]
	Therefore, we have the following proposition.

\begin{prop}
	If $\beta\in\Phi^+$ is a long root, then 
	\begin{equation}\label{final formula for ebeta pi}
		[E_\beta, \pi]=E_\beta \wedge H_\beta + \sum\limits_{\substack{\alpha\in\Phi_{s_\beta}\\ \alpha\neq \beta}}\lambda_\alpha c_{\beta, -\alpha} E_\alpha \wedge E_{\beta-\alpha}.
	\end{equation}
\end{prop}
%

\subsection{Computations of Zambon coisotropic subalgebras}  We are now in the position to understand $\u_\beta=[E_\beta, \pi]^\sharp \g^*$ where $\beta$ is a long root.  We begin by computing $[E_\beta, \pi]^\sharp \g^*$ using the results of the previous section.  We then  prove $(\u_\beta)_\Delta \oplus \u_\beta^\perp$ is in $\CL(\g\oplus\g)$ for all long roots $\beta$ which recovers Theorem \ref{Zambon's Theorem}.

\begin{prop}\label{Prop11}
	If $\beta\in \Phi^+$ is a long root, then
		\begin{equation}\label{doesn't actually need a label}
			\u_\beta:=[E_\beta, \pi]^\sharp \g^* = \C H_\beta + \bigoplus\limits_{\alpha\in\Phi_{s_\beta}} \g_\alpha
		\end{equation}
	and
		\begin{equation}\label{second eqn}
			\u_{-\beta}:=[E_{-\beta}, \pi]^\sharp \g^* = \C H_\beta + \bigoplus\limits_{\alpha\in\Phi_{s_\beta}} \g_{-\alpha}.
		\end{equation}
\end{prop}

\begin{proof}
	Let $\{E_\alpha^*|\alpha\in\Phi\}\cup \{H_\beta^*\}\cup \{f_i|2\leq i\leq \text{rank}(\g)\}$ be a basis of $\g^*$ where the $f_i$ are linear functions on $\h$ such that $f_i(H_\beta)=0$ for all $i$, and $H_\beta^*$ vanishes on $\n$ and on $\n_-$ and evaluates as $1$ against $H_\beta$.
	From equation (\ref{final formula for ebeta pi}), 
	$[E_\beta, \pi]^\sharp(f_i)=0$ for all $i$ and $[E_\beta, \pi](H_\beta^*)=-E_\beta.$  Since $[E_\beta, \pi]$  involves only positive roots, $[E_\beta, \pi]^\sharp (E_{-\gamma}^*)=0$ for all $\gamma\in\Phi^+$. Furthermore, $[E_\beta, \pi]^\sharp (E_\beta^*)=H_\beta$. Finally, $\alpha\in\Phi_{s_\beta}\setminus \{\beta\}$ implies $[E_\beta, \pi]^\sharp (E_\alpha^*)=E_{\beta-\alpha}$ and $[E_\beta, \pi]^\sharp (E_{\beta-\alpha}^*)=-E_\alpha$.  Therefore,
		\begin{eqnarray*}
			[E_\beta, \pi]^\sharp \g^* & =\C H_\beta + \C E_\beta +\sum\limits_{\alpha\in\Phi_{s_\beta}\setminus \{\beta\}} \C E_{\alpha}  \\	
				&=\C H_\beta +\sum\limits_{\alpha\in\Phi_{s_\beta}} \C E_\alpha.
		\end{eqnarray*}
		
	Equation (\ref{second eqn}) follows in a similar fashion.
\end{proof}

We now show that $\u_\beta$ and $\u_{-\beta}$ are coisotropic subalgebras.  The following proposition also proves that Zambon's coisotropic subalgebras are a special case of those described in Theorem \ref{Main Lemma about V+(un,0)+(0,vn-) being coisotropic}.

\begin{prop}\label{Zambon Claim}
	Let $\beta\in\Phi^+$ be a long root, then 
			\[
				(\u_\beta)_\Delta+\u_\beta^\perp=(\C H_\beta)_\Delta+(\C H_\beta)_{-\Delta}^\perp+(\n, 0)+(0,s_\beta \cdot \n_-)=\mathfrak{l}_{\C H_\beta, e, s_\beta}
			\]
	 and 
	
			\[
				(\u_{-\beta})_\Delta+\u_{-\beta}^\perp=(\C H_\beta)_\Delta+(\C H_\beta)_{-\Delta}^\perp+(v\cdot\n, 0)+(0, \n_-)=\mathfrak{l}_{\C H_\beta, s_\beta, e}.
			\]
	Furthermore, both of these subspaces are in $\CL(\g\oplus\g)$.  In particular, since $\Phi_w\cap\Phi_e=\emptyset$ for all $w\in W$, Zambon's coisotropic subalgebras are a special case of the form of Theorem \ref{Main Lemma about V+(un,0)+(0,vn-) being coisotropic}.

\end{prop}

	\begin{proof}
		By Proposition \ref{Prop11}, $\u_\beta=\C H_\beta + \bigoplus_{\alpha\in\Phi_{s_\beta}} \g_\alpha$. Now, compute $\u_\beta^\perp$ in $\g^*$. First, note that
			\begin{eqnarray*}
				\u_\beta^\perp=\{(H+ \sum_{\gamma\in\Phi^+} c_\gamma E_\gamma, -H+ \sum_{\eta\in\Phi^+} b_\eta E_{-\eta})\in \g^* | \\ \langle (x,x), (H+ \sum_{\gamma\in\Phi^+} c_\gamma E_\gamma, -H+ \sum_{\eta\in\Phi^+} b_\eta E_{-\eta}) \rangle=0 \text{ for all } x\in \u_\beta \}.
			\end{eqnarray*}
	It is easy to check that $\u_\beta^\perp$ contains
			\begin{equation}\label{dads equation}
			(V^\perp)_{-\Delta}+(\n,0)+\bigoplus_{\alpha\in \Phi_{s_\beta}^c}( 0, \g_{-\alpha})
			\end{equation}
		 where $V=\C H_\beta$. 
		 Furthermore, $\dim(\u_\beta^\perp)=\dim(\g^*)-\dim(\u_\beta)$ and it follows that $\u_\beta^\perp$ is exactly equal to the set in equation (\ref{dads equation}).
		Hence,
			\begin{align*}
				(\u_\beta)_\Delta\oplus\u_\beta^\perp & =V_\Delta + (V^\perp)_{-\Delta} + \bigoplus\limits_{\alpha\in \Phi_{s_\beta}} (\g_\alpha)_\Delta + (\n, 0)+ \bigoplus\limits_{\alpha\in \Phi_{s_\beta}^c} (0,\g_{-\alpha}) \\
				& = V_\Delta + (V^\perp)_{-\Delta}+(e\cdot \n, 0)+(0, s_\beta\cdot\n_-) \\
				& = \fl_{V,e,s_\beta}
			\end{align*}
		where the next to last equality follows from the proof of Theorem \ref{Main Lemma about V+(un,0)+(0,vn-) being coisotropic}.  In particular,for any $v\in W$, $\Phi_e\cap \Phi_v=\emptyset$, so by Theorem \ref{Main Lemma about V+(un,0)+(0,vn-) being coisotropic}, $(\u_\beta)_\Delta \oplus\u_\beta^\perp\in \CL(\g\oplus\g)$.  Furthermore, $\u_\beta $ is a coisotropic subalgebra of $\g$ by Remark \ref{remark to add}.
		The statement for $\u_{-\beta}$ can be seen similarly.
	\end{proof}

Combining Propositions \ref{Prop11} and \ref{Zambon Claim}, we have 

\begin{theorem}\label{theorem with calculation for ubeta}
	If $\beta\in\Phi^+$ is a long root, then 
	\begin{equation}
		\u_\beta:=[E_\beta, \pi]^\sharp \g^* = \C H_\beta + \bigoplus\limits_{\alpha\in\Phi_{s_\beta}} \g_\alpha
	\end{equation}
	and 
		\begin{equation}
			\u_{-\beta}:=[E_{-\beta}, \pi]^\sharp \g^* = \C H_\beta + \bigoplus\limits_{\alpha\in\Phi_{s_\beta}} \g_{-\alpha}
		\end{equation}
	are a special case of the coisotropic subalgebras of Theorem \ref{Main Lemma about V+(un,0)+(0,vn-) being coisotropic}.
\end{theorem}
In particular, we have recovered Zambon's theorem, Theorem \ref{Zambon's Theorem}, about coisotropic subalgebras and have shown that his coisotropic subalgebras are a special case of Theorem \ref{Main Lemma about V+(un,0)+(0,vn-) being coisotropic}. In \cite{Zambon}, Zambon only explicitly describes his coisotropic sublagebras in the case of classical Lie algebras.  Note that Theorem \ref{theorem with calculation for ubeta} applied to all semisimple Lie algebras, not just the classical Lie algebras.

\begin{remark}
	By the proof of Proposition \ref{final formula for ebeta pi}, the elements of $\Phi_{s_\beta}\setminus \{\beta\}$ come in pairs.  Therefore, $|\Phi_{s_\beta}|$ is odd, and $\dim(\u_\beta)=1+|\Phi_{s_\beta}|$ is even.  Thus, Zambon coisotropic subalgebras are all even dimensional.
\end{remark}

\begin{remark}
As noted in \cite{Zambon}, there are odd dimensional coisotropic subalgebras.  This follows easily from Theorem \ref{Main Lemma about V+(un,0)+(0,vn-) being coisotropic}.
\end{remark}

	\bibliographystyle{plain}
 \bibliography{Bibliography}           

\end{document}